\documentclass{amsart}
\usepackage{amsfonts,amssymb,amscd,amsmath,enumerate,verbatim,calc}
\usepackage[all]{xy}

\newcommand{\CM}{Cohen-Macaulay}

\newcommand{\m}{\mathfrak{m} }

\newcommand{\q}{\mathfrak{q} }

\newcommand{\C}{\mathcal{C} }

\newcommand{\Fc}{\mathcal{F} }
\newcommand{\rt}{\rightarrow}

\newcommand{\Gc}{\mathcal{G} }
\newcommand{\p}{\mathfrak{p} }

\newcommand{\Min}{\operatorname{Min}}

\newcommand{\mSpec}{\operatorname{m-Spec}}
\newcommand{\height}{\operatorname{height}}
\newcommand{\rad}{\operatorname{rad}}
\newcommand{\Ass}{\operatorname{Ass}}

\theoremstyle{plain}

\newtheorem{theorem}{Theorem}[section]

\newtheorem{proposition}[theorem]{Proposition}

\theoremstyle{definition}

\newtheorem{s}[theorem]{}
\newtheorem{remark}[theorem]{Remark}

\theoremstyle{remark}

\begin{document}

\title[$R_n$ ]{Local cohomology of ideals  and the $R_n$ condition of Serre}
\author{Tony~J.~Puthenpurakal}
\date{\today}
\address{Department of Mathematics, IIT Bombay, Powai, Mumbai 400 076}

\email{tputhen@gmail.com}
 \begin{abstract}
Let $R$ be a regular ring of dimension $d$ containing a field $K$ of characteristic zero. If $E$ is an $R$-module let $\Ass^i E = \{ Q \in \Ass E \mid \height Q = i \}$.
Let $P$ be a prime ideal in $R$ of height $g$. We show that if $R/P$ satisfies Serre's condition $R_i$ then $\Ass^{g+i+1}H^{g+1}_P(R)$ is a finite set. As an application of our techniques  we prove that if $P$ is a prime ideal in $R$ such that $(R/P)_\mathfrak{q}$ is regular for any non-maximal prime ideal $\mathfrak{q}$ then $H^i_P(R)$ has finitely many associate primes for all $i$.
\end{abstract}
 \maketitle
\section{introduction}
Throughout this paper $R$ is a commutative Noetherian ring. If $M$ is an $R$-module and if $I$ is an ideal in $R$, we denote by $H^i_I(M)$ the  $i^{th}$ local cohomology module of $M$ with respect to $I$.

The following conjecture  is due to Lyubeznik  \cite{Lyu-3};

\s \label{conj}\textbf{Conjecture:} If $R$ is a regular ring, then each local cohomology module $H^i_I(R)$ has
finitely many associated prime ideals.

There are many cases where this conjecture is true: by work of Huneke and Sharp \cite{HuSh},
for regular rings $R$ of prime characteristic; by work of Lyubeznik, for regular local and
affine rings of characteristic zero \cite{Lyu-1},  for unramified regular local rings of mixed
characteristic \cite{Lyu-4}. It is also true for smooth $\mathbb{Z}$-algebras by work of Bhatt et.al \cite{BB}.

In \cite{Lyu-3} Lyubeznik especially asked whether \ref{conj} is valid for a regular ring $R$ containing a field of characteristic zero. It is easy to give examples where existing techniques, to show finiteness of associate primes of local cohomology modules, fail.

For the rest of the paper assume that $R$ is regular and contains a field of characteristic zero.
For simplicity we assume that $\dim R = d$ is finite.
By Grothendieck vanishing theorem $H^i_I(R) = 0$ for all $i > d$, see \cite[6.1.2]{BSh}. In general for a Noetherian ring $T$ of dimension $d$
the set $\Ass_R(H^d_I(T))$ is  finite, see  \cite[3.11]{BRS} (also see \cite[2.3]{TM}). If $T$ is a regular ring of dimension $d$ and $I$ is an ideal in $T$ then using the Hartshorne-Lichtenbaum theorem, cf. \cite[14.1]{a7}  it is easy to prove that
$$\Ass_R (H^d_I(T)) = \left\{ P  \left|  P \in \Min T/I \ \text{and} \ \height P = d   \right.   \right \}.  $$
In \cite{P},  we proved that $H^{d-1}_I(R)$ is a finite set if $R$ is an excellent ring. We also proved that to prove Lyubeznik's conjecture for excellent regular rings containing a field of characteristic zero it is sufficient to prove
$\Ass H^{g+1}_I(R)$ is a  finite for all ideals of height $g$.
Practically nothing is known for rings which are not excellent. We also note that in examples of singular rings $S$ with ideals $I$ such that $\Ass H^i_I(S)$ is an infinite set the ring $S/I$ is regular, see \cite{S}, \cite{K} and \cite{SS}.  If $T$ is a regular ring  and $I$ is an ideal such that $T/I$ is a regular ring then it is not difficult to show that $\Ass H^i_I(T)$ is a finite set for all $i$. In this paper we prove
\begin{theorem}
\label{main} Let $R$ be a regular ring of dimension $d$, containing a field of characteristic zero. Let $P$ be a prime ideal in $R$ such that  $(R/P)_\mathfrak{q}$ is regular for any non-maximal prime ideal $\mathfrak{q}$.  Then $H^i_P(R)$ has finitely many associate primes for all $i$.
\end{theorem}
It is not difficult to prove the above result when $R$ is excellent. The case when $R$ is not excellent requires new ideas.

If $E$ is an $R$-module let $\Ass^i E = \{ Q \in \Ass E \mid \height Q = i \}$. We show
\begin{theorem}
\label{main-r1} Let $R$ be a regular ring of dimension $d$,  containing a field of characteristic zero. Let $P$ be a prime ideal in $R$ such that $R/P$ satisfies Serre's condition $R_i$. Then
 $\Ass^{g+i+1}H^{g+1}_P(R)$ is a finite set.
\end{theorem}

Here is an overview of the contents of this paper. In section two we consider the flat extension $R \rt R[[X]]_X$. This helps when $R$ contains a countable field $K$. Note $T = R[[X]]_X$ contains the field $K[[X]]_X$ which is uncountable. We show that to prove Theorems \ref{main} and \ref{main-r1} it suffices to assume that $R$ contains an uncountable field. In section three we describe a construction which is important to us. Finally in section four we prove Theorems \ref{main} and \ref{main-r1}.
\section{Passage to a ring containing an uncountable field}
In our arguments we need to assume that $R$ contains an uncountable field. When this is not the case we consider the flat extension $R \rt R[[X]]_X$.
Set $S = R[[X]]$ and $T = S_X = R[[X]]_X$, i.e., the ring obtained by inverting $X$.

\begin{remark}\label{passage}
(i) If $R$ contains a countable field $K$ then note $K[[X]]$ is a subring of $ S$ and so $K[[X]]_X $ is a subring of $T$. The field $K[[X]]_X$ is uncountable. Thus $T$ contains an uncountable field.

(ii) If $R$ is regular (\CM) then so is $S$, see \cite[2.2.13]{BH} and \cite[2.1.9]{BH}. Therefore $T$ is also regular (\CM).

\end{remark}

The following properties were proved in  our paper \cite[section 3]{P}.
\begin{proposition}(with hypotheses as above)
\begin{enumerate}[\rm (1)]
\item
$T$ is a faithfully flat $R$-algebra.
\item
If $P$ is a prime ideal in $R$ then $PT$ is a prime ideal in $T$ and $PT \cap R = P$. Conversely if $P$ is a prime ideal in $T$ then there is a prime ideal $Q$ of $R$ with $P = QT$.
Furthermore $\height P = \height PT$.
\item
$\dim R = \dim T$.
\item
Let $M$ be a $R$-module. The mapping defined by
\begin{align*}
\psi \colon \Ass_R(M) &\rt \Ass_T (M\otimes_R T) \\
               \p &\rt \p T
\end{align*}
is a bijection.
\item
 $\psi$ maps $\Ass^i_R(M)$ bijectively to $\Ass^i_T(M\otimes_R T)$.
\end{enumerate}
\end{proposition}
We will need the following{
\begin{proposition}
\label{ex-r1}
(with hypotheses as above). Also assume $R$ is \CM.
Let $P$ be a prime in $R$ such that $R/P$ satisfies $R_i$. Then $T/PT$ satisfies $R_i$.
\end{proposition}
\begin{proof}
We note that $T$ is also \CM. So both $R, T$ are universally catenary.
Let $\height P = g$.
Let $Q$ be  a prime ideal in $T$ containing  $PT$ with $\height Q = g + i$. Let $Q = \q T$. Then $\height \q = g + i$ and $\q$ contains $P$.
So $(R/P)_\p$ is regular. The extension $(R/P)_\q \rt (T/PT)_Q$ is flat with field as a fiber. By \cite[23.7]{Mat} it follows that $(T/PT)_Q$ is regular.
\end{proof}
\section{Construction}
In this section we describe a construction that we use to prove our results.

\s \label{const-K} Let $R$ be a Noetherian ring containing an uncountable field $K$. Assume \\
$\mSpec(A) = \{ \m_n  \mid n \geq 1 \}$ and that $\height \m_n = d$ for all $n$.

Let $$\C  = \{ \bigcap_{\m \in \Fc }\m \mid  \Fc \ \text{is an infinite subset of } \ \mSpec(A)\}.$$
As $R$ is Noetherian $\C$ has maximal elements. Let $Q$ be a maximal element in $\C$.
Then $Q = \bigcap_{\m \in \Fc }\m$ for some infinite subset $\Fc$ of $\mSpec(A)$. We note that as $Q$ is a maximal element in $\C$ we have
 $Q = \bigcap_{\m \in \Gc }\m$ for any infinite subset $\Gc$ of $\Fc$.

  Consider the multiplicatively closed set
\[
S = R \setminus \bigcup_{\m \in \Fc} \m.
\]
Set $T = S^{-1}R$.

The following result  from \cite[4.4]{P} gives information about primes in $T$.
\begin{proposition}
\label{prop-const}(with hypotheses as in \ref{const-K}).
We have
\begin{enumerate}[\rm (1)]
\item
If $\p T$ is a prime in $T$, where $\p$ is a prime in $R$, then $\p \subseteq \m$ for some $\m \in \Fc$.
\item
If $\m_i, \m_j \in \Fc$ then
$\m_i T \nsubseteq \m_j T$  if $\m_i \neq \m_j$..
\item
$\m T$  for $\m \in \Fc$ are distinct maximal ideals of $T$.
\item
$\mSpec(T) = \{ \m T\}_{\m \in \Fc}$.
 \end{enumerate}
\end{proposition}

We will also need the following intersection result from \cite[4.5]{P}.
\begin{proposition}\label{prop-const-2}(with hypotheses as in \ref{const-K}).
Let $\Lambda$ be any subset of $\Fc$ ( possibly infinite). Set
$$ U = \bigcap_{\m \in \Lambda} \m \quad \text{and} \ V = \bigcap_{\m \in \Lambda} \m T. $$
Then $U T = V$.
\end{proposition}

\s The following result describes the properties of $T$ that we need.
\begin{theorem}\label{prop}
(with hypotheses as in \ref{const-K}). We have
\begin{enumerate}[\rm (1)]
\item
$QT = \rad T$.
\item
Let $\Lambda$ be any infinite subset of $\Fc$. Then
$$ QT = \bigcap_{\m \in \Lambda} \m T.$$
\item
$QT$ is a prime ideal of $T$.
\item
$\height \m T = \height \m = d$ for all $\m \in \Fc$.
\item
If $f\notin QT$ then $f$ belongs to utmost finitely many maximal ideals of $T$.
\end{enumerate}
\end{theorem}
\begin{proof}
(1) This follows from  \ref{prop-const} and \ref{prop-const-2}.

(2)We note that by maximality of $Q$ in $\C$ we have $ Q = \bigcap_{\m \in \Lambda} \m.$
The result now follows from \ref{prop-const-2}.

(3) Note $QT$ is a radical ideal. So $Q T= P_1\cap P_2 \cap \cdots \cap P_s$ for some primes  $P_i$ with $s \geq 1$ and that $P_i \nsubseteq P_j$ if $i \neq j$. Suppose if possible $s \geq 2$. Let $\m T$ be a maximal ideal of $T$. Then note that $\m T \supseteq P_i$ for some $i$. Let
$$\Gamma_i = \{ \m T \mid \m T \supseteq  P_i \}.$$
Then
$$\mSpec(T) = \bigcup_{i = 1}^{s} \Gamma_i.$$
It follows that $\Gamma_i$ is infinite for some $i$. Without any loss of generality we assume $i = 1$.
Then note that $P_1 \subseteq \bigcap_{\m T\in \Lambda_1} \m T = QT$.
It follows that $P_2 \supseteq P_1$, contradicting the choice of $P_2$. Thus $QT$ is a prime ideal in $T$.

(4)This is trivial.

(5) Suppose if possible there exists an infinite subset $\Gc$ of $\mSpec(T)$ such that $f \in \m T$ for every $\m T \in \Gc$. It follows that
$f \in \cap_{\m T \in \Gc} \m T$. But by construction the latter module is $QT$. This is a contradiction.
\end{proof}

\s \label{unit} If $f \notin QT$ we describe a process of making $f$ a unit. Let $\Gc$ be the finite set consisting of ALL the maximal ideals $\m $  of $T$ such that $f \in \m$. Set
$$ W = T \setminus \bigcup_{\m \in \mSpec(T)\setminus \Gc}\m   \quad \text{and} \quad L = W^{-1}T.$$
We note that $f$ is a unit in $L$. Also if $\m \in \mSpec(T)\setminus \Gc $ then $\m L$ is a maximal ideal in $L$. Furthermore $Q L = \rad L$.
\section{Proof of Theorems \ref{main} and \ref{main-r1}}
In this section we give proofs of Theorems \ref{main} and \ref{main-r1}.
We first give
\begin{proof}[Proof of Theorem \ref{main-r1}]
By our results in section two we may assume that $R$ contains an uncountable field of characteristic zero. Suppose $\Ass^{g+i+1}H^{g+1}_P(R)$ is an  infinite set. Let
 $$\Ass^{g+i+1}H^{g+1}_P(R) = \{ P_n \mid n \geq 1 \}.$$
  Consider $W = R \setminus \bigcup_{n \geq 1}P_n$ and set $S = W^{-1}R$. Next we do the construction as described in section three.
  Let $T$ be the ring obtained with $Q = \rad T$. We note that $\dim T = g + i + 1$. We also have $\height Q \leq g + i$. As $R/P$ satisfies $R_i$ it follows that $R_Q/PR_Q$ is a regular ring. It follows that $P_Q = (u_1, \ldots, u_g)$. So there exists $f \notin Q$ such that $P_f = (u_1, \ldots, u_g)$. By process of making $f$ a unit, see \ref{unit}, we get a further localization $L$ of $T$ such that $PL$ is generated by
  $g$ elements. So $H^{g+1}_{PL}(L) = 0$ and all maximal ideals of $L$ are associate primes of $H^{g+1}_{PL}(L) $.  This is a contradiction. Thus $\Ass^{g+i+1}H^{g+1}_P(R)$ is a  finite set.
\end{proof}

Next we give
\begin{proof}[Proof of Theorem \ref{main}]
By our results in section two we may assume that $R$ contains an uncountable field of characteristic zero. Suppose $\Ass H^i_P(R)$ is infinite for some $i$.  Set $M = H^i_P(R)$. We may assume that there exists $r$ such that $M$ contains infinitely many associate primes $P$ of height $r$ and there exists only finitely many associate prime of $M$ with height $< r$, Let $\{ P_n \}_{n \geq 1}$ be infinitely many associate primes of $M$ of height $r$. We note that necessarily $r > g = \height P$. Consider $W = R \setminus \bigcup_{n \geq 1}P_n$ and set $S = W^{-1}R$. Next we do the construction as described in section three.
  Let $T$ be the ring obtained with $Q = \rad T$. We note that $\dim T = r$ and $\height Q < r$. By our assumption $(R/P)_Q$ is a regular local ring. It follows that $P_Q = (u_1, \ldots, u_g)$. So there exists $f \notin Q$ such that $P_f = (u_1, \ldots, u_g)$. By process of making $f$ a unit, see \ref{unit},  we get a further localization $L$ of $T$ such that $PL$ is generated by
  $g$ elements. So $H^{r}_{PL}(L) = 0$ and all maximal ideals of $L$ are associate primes of $H^{r}_{PL}(L) $.  This is a contradiction. Thus $\Ass H^{i}_P(R)$ is a  finite set.
\end{proof}
%\section*{Acknowledgements}

\end{document}